\documentclass[12pt]{amsart} 
\usepackage[mathscr]{eucal} 
\usepackage{amsmath,amsfonts} 
\parskip=\smallskipamount 
\newtheorem{theorem}{Theorem}[section] 
\newtheorem{proposition}[theorem]{Proposition} 
 
\newtheorem{lemma}[theorem]{Lemma} 
\theoremstyle{definition} 
\newtheorem{definition}[theorem]{Definition}

\theoremstyle{fancyproclaim}

\makeatletter 
\@addtoreset{equation}{section} 
\makeatother

\newcommand{\CC}{{\mathbb C}} 
\newcommand{\NN}{{\mathbb N}}

\newcommand{\RR}{{\mathbb R}}

\newcommand{\cB}{{\mathcal B}} 
\newcommand{\cC}{{\mathcal C}} 
\newcommand{\cD}{{\mathcal D}}

\newcommand{\cH}{{\mathcal H}}

\newcommand{\cL}{{\mathcal L}}

\newcommand{\cS}{{\mathcal S}} 
\newcommand{\cT}{{\mathcal T}}

\newcommand{\dom}{\operatorname{Dom}}

\newcommand{\ran}{\operatorname{Ran}} 
 
\newcommand{\ra}{\rightarrow}

\let\phi=\varphi 
 
\renewcommand{\ker}{\operatorname{Ker}}

\begin{document} 
\title[Selfadjoint Extensions]{Notes on (the Birmak--Kre\u\i n--Vishik theory on) selfadjoint extensions 
of semibounded symmetric operators}\thanks{\textit{Note of the second named author}: 
This is a manuscript that was circulated as the first part of the preprint 
\textit{"Two papers on selfadjoint extensions of symmetric semibounded operators", INCREST 
Preprint Series, July 1981, Bucharest, Romania,} but never published. In this \LaTeX\ typeset version, 
only typos and a few inappropriate formulations have been corrected, with respect to the original manuscript.
I decided to post it on arXiv since, taking into account recent articles, the results are still of current interest.
Tiberiu Constantinescu died in 2005.}
 
 \date{\today}
  
\author[T. Constantinescu]{Tiberiu Constantinescu} 
\address{Tiberiu Constantinescu: Department of Mathematics, National Institute for Scientific and Technical Creation,
Bd. P\u acii 220, 79622, Bucharest, Romania}
  
\author[A. Gheondea]{Aurelian Gheondea} 
\address{Aurelian Gheondea: Department of Mathematics, National Institute for Scientific 
and Technical Creation, Bd. P\u acii 220, 79622, Bucharest, Romania, \textit{current addresses:}
Department of Mathematics, Bilkent University, 06800 Bilkent, Ankara, 
Turkey, \emph{and} Institutul de Matematic\u a al Academiei Rom\^ane, C.P.\ 
1-764, 014700 Bucure\c sti, Rom\^ania} 
\email{aurelian@fen.bilkent.edu.tr \textrm{and} A.Gheondea@imar.ro} 

\begin{abstract} We give an explicit and versatile parametrization of all positive 
selfadjoint extensions
of a densely defined, closed, positive operator. In addition, we identify the Friedrichs extension
by specifying the parameter to which it corresponds.
\end{abstract} 

\subjclass[2010]{}
\maketitle 

\section{Preliminaries}

Consider a closed, densely defined, symmetric operator $S$ on a Hilbert space $\cH$. 
When dealing with such an 
operator, the main problem is to extend it to a selfadjoint one. A complete result to this problem was 
given by J.~von Neumann.

A second and more difficult problem is to find all the semibounded selfadjoint 
extensions of a given semibounded symmetric operator $S$. For such an operator, the existence
of a semibounded selfadjoint extension having the same maximal semibound was solved by 
Friedrichs. The important step in this direction was done by M.G.~Kre\u\i n \cite{Krein}, and immediately
after that by M.S.~Birman and M.I.~Vishik, 
and this is what is called the Birman--Kre\u\i n--Vishik theory.

A possible approach involving quadratic forms on Hilbert spaces was recently pointed out by 
A.~Alonso and B.~Simon \cite{AlonsoSimon}.

Stimulated by this kind of investigations, the aim of this paper is to give a new and easy to 
handle parametrization of the set of all semibounded selfadjoint extensions and, 
simultaneously, new proofs to classical results are obtained.

In the
following, block-matrix representations are used with respect to appropriate
orthogonal decompositions of Hilbert spaces. 
As a starting point, we need only two results concerning completing matrix contractions. 
The first one is the Sz.-Nagy--Foia\c s Lemma \cite{NagyFoias}.

\begin{lemma}\label{l:fn}
Let $T=\left[\begin{array}{ll} T_1 & X \\ 0 & T_2\end{array}\right]$. Then $T$ is a contraction if and 
only if $T_1$, $T_2$ are contractions and $X=D_{T^*_1}CD_{T_2}$ with $C$ a contraction 
$\cD_{T_2}\ra \cD_{T_1^*}$.
\end{lemma}
Here, for a given contraction $C\colon\cH_1\ra\cH_2$ and Hilbert spaces $\cH_1$ and $\cH_2$, 
we denote by $D_C=(I_{\cH_1}-C^*C)^{1/2}$,
the \emph{defect operator} of $C$, and by $\cD_C=\overline{D_C\cH_1}$, the 
\emph{defect space} of $C$.

The second one is a recently obtained result of Gr.~Arsene and 
A.~Gheondea \cite{ArseneGheondea}. Let 
$A_\mathrm{r}=[A\quad B]$ and $A_\mathrm{c}=\left[\begin{matrix} A \\ C\end{matrix}\right]$ be two
contractions. Then, by Lemma~\ref{l:fn} there exist unique contractions $\Gamma_1$ valued in 
$\cD_{A^*}$ and $\Gamma_2$ defined in $\cD_A$ such that 
$A_\mathrm{r}=[A\quad D_{A^*}\Gamma_1]$ and $A_\mathrm{c}=\left[\begin{matrix} A \\ 
\Gamma_2D_A\end{matrix}\right]$.

\begin{lemma}\label{l:ag}
There exists a bijective correspondence between the sets $\{T=\left[\begin{array}{ll} A & B \\ C & X
\end{array}\right]\mid T\mbox{ is a contraction}\}$ and $\{\Gamma\colon\cD_{\Gamma_1}\ra
\cD_{\Gamma_2^*}\mid \Gamma\mbox{ is a contraction}\}$ given by the formula
\begin{equation*}
T(\Gamma)=\left[\begin{array}{ll} A & D_{A^*}\Gamma_1 \\ \Gamma_2 D_A & 
-\Gamma_2 A^* \Gamma_1+D_{\Gamma_2^*}\Gamma D_{\Gamma_1}
\end{array}\right]. 
\end{equation*}
\end{lemma}

We make use of the old idea of M.G.~Kre\u\i n to introduce a special kind of Cayley transform so we 
first make some preparations.

\section{The Cayley Transform}

Let us consider a closed, densely defined, 
symmetric operator $S_0$ acting on a Hilbert space $\cH$. Suppose
also that it is bounded from below, i.e. there exists $a\in\RR$ such that
\begin{equation}\label{e:sm}
\langle S_0h,h\rangle\geq a\|h\|^2,\mbox{ for all }h\in\dom(S_0).
\end{equation}

If $m(S_0)$ is the largest real number $a$ such that \eqref{e:sm} holds true then for every 
$m_0\leq m(S_0)$ the operator $S=S_0-m_0I$ is positive and it is easy to check that, in order to 
find all selfadjoint extensions of $S_0$ bounded from below by $m_0$, we have to find all
positive selfadjoint extensions of $S$, \cite{Krein}.

Consider a densely defined positive operator $S$, that is, \eqref{e:sm} holds with $a=0$. Then 
$S$ is closable and hence, without restricting the generality, we can assume that $S$ is closed. 
For every $h\in\dom(S)$, here and in throughout $\dom(S)$ denotes the domain of $S$, we get
\begin{align}\label{e:ips} \|(I+S)h\|^2 & = \|h\|^2+2\langle Sh,h\rangle+\|Sh\|^2 \\
& \geq \|h\|^2-2\langle Sh,h\rangle+\|Sh\|^2=\|(I-S)h\|^2.\nonumber
\end{align}
It follows that $I+S$ is one-to-one and $\ran(I+S)$, the range of $I+S$, is closed. These all 
enable us to define the operator
\begin{equation}\label{e:cs}
C(S)=T\colon\ran(I+S)\ra\cH,\quad T=(I-S)(I+S)^{-1},
\end{equation}
 and this is what we call the \emph{Cayley transform} of $S$. 
 By means of \eqref{e:ips} one can easily prove
 that $T\colon\dom(T)\ra\cH$ is a contraction, hence bounded. Moreover, $T$ is symmetric, since
 \begin{align*}
 \langle T(I+S)h,(I+S)g\rangle & =\langle (I-S)h,(I-S)g\rangle\\
 & =\langle h,g\rangle-\langle Sh,g\rangle +\langle h,Sg\rangle-\langle Sh,Sg\rangle \\
 & =\langle h,g\rangle -\langle h,Sg\rangle +\langle Sh,g\rangle -\langle Sh,Sg\rangle \\
 & = \langle (I+S)h,(I-S)g\rangle = \langle (I+S)h,T(I+S)g\rangle,\quad h,g\in\dom(S). 
 \end{align*}
 Since
\begin{equation*}
(I+T)(I+S)h=(I+S)h+(I-S)h=2h,\quad h\in\dom(S),
\end{equation*}
$I+T$ is one-to-one, and $\ran(I+T)$ is dense in $\cH$.

Conversely, suppose that $T$ is a symmetric contraction with $\dom(T)$ closed and such that 
$\overline{\ran(I+T)}=\overline{(I+T)\dom(T)}=\cH$. Then $I+T$ is one-to-one, hence one can introduce the operator
\begin{equation}\label{e:ct}
C^{-1}(T)=S\colon\ran(I+T)\ra\cH,\quad S=(I-T)(I+T)^{-1},
\end{equation}
and from the assumptions on $T$ one can prove easily that $S$ is a positive closed and densely
defined operator on $\cH$. We have proven

\begin{lemma}[M.G.~Kre\u\i n] \label{l:k1}
The Cayley transform is bijective between the set of all positive closed and densely defined
operators $S$ on $\cH$ and the set of all symmetric contractions 
$T\colon\dom(T)\subseteq\cH\ra\cH$, with $\dom(T)$ closed and $\overline{(I+T)\dom(T)}=\cH$.
\end{lemma}

The following result is also essential for our approach.

\begin{lemma}[M.G.~Kre\u\i n]\label{l:k2} 
For a given positive, densely defined, and 
closed operator $S$ on $\cH$, the Cayley transform \eqref{e:cs}
is bijective between the sets $\cS$ and $\cT$
\begin{align*}
\cS & =\{\widetilde S\mid \widetilde S\mbox{ is positive selfadjoint, } \widetilde S|\dom(S)=S\},\\
\cT & = \{\widetilde T\mid \widetilde T\mbox{ is a symmetric contraction on }\cH,\ 
\widetilde T|\dom(T)=T\}.
\end{align*}
\end{lemma}
\begin{proof}
It is a classical result, see \cite{Kato}, that a positive operator $R$ is selfadjoint if and only if 
$(I+R)\dom(R)=\cH$. Finally, use Lemma~\ref{l:k1}.
\end{proof}

\begin{definition}[\cite{Kato}, \cite{ReedSimon}] Suppose $R_1$ and $R_2$ are two positive
selfadjoint operators on $\cH$. Then $R_1\leq R_2$ means
\begin{equation}\label{e:aud}
 \dom(R_2^{1/2})\subseteq\dom(R_1^{1/2})\mbox{ and } 
\|R_1^{1/2}\xi\|\leq \|R_2^{1/2}\xi\|\mbox{ for all }\xi\in\dom(R_2^{1/2}).
\end{equation}
\end{definition}

\begin{lemma}\label{l:ord}
If $R_1$ and $R_2$ are two positive selfadjoint operators on $\cH$ then $R_1\leq R_2$ if and only
if $C(R_1)\geq C(R_2)$.
\end{lemma}

\begin{proof}
We use essentially the following result from \cite{Kato}, VI.2, Theorem 2.21,
\begin{equation}\label{e:aud}
R_1\leq R_2\mbox{ if and only if } (I+R_1)^{-1}\geq (I+R_2)^{-1},
\end{equation}
the order from the right hand side being the usual one for bounded selfadjoint operators.

Suppose $R_1\leq R_2$. What we have to prove is that for every $\xi\in\cH$ it holds
\begin{equation}\label{e:pai}
\langle (I-R_1)(I+R_1)^{-1}\xi,\xi\rangle \geq \langle (I-R_2)(I+R_2)^{-1}\xi,\xi\rangle.
\end{equation}
To this end, for a fixed $\xi\in\cH$ there exist two uniquely determined vectors $h\in\dom(R_1)$ 
and $g\in\dom(R_2)$ such that
\begin{equation}\label{e:iau}
(I+R_1)h=\xi=(I+R_2)g,
\end{equation}
therefore \eqref{e:pai} is equivalent with
\begin{equation*}
\langle (I-R_1)h,(I+R_1)h\rangle \geq \langle (I-R_2)g,(I+R_2)g\rangle,
\end{equation*}
and this holds if and only if
\begin{equation}\label{e:noh}
\|h\|^2-\|R_1h\|^2\geq \|g\|^2-\|R_2g\|^2.
\end{equation}
Making use of \eqref{e:iau}, it follows that
\begin{equation*}
\|(I+R_1)h\|^2=\|(I+R_2)g\|^2,
\end{equation*}
hence
\begin{equation}\label{e:adog}
\|R_2g\|^2-\|R_1h\|^2=2\bigl(\langle R_1h,h\rangle-\langle R_2g,g\rangle\bigr)
+\bigl(\|h\|^2-\|g\|^2\bigr).
\end{equation}
From \eqref{e:adog} we get that \eqref{e:noh} holds if and only if
\begin{equation*}
\langle (I+R_1)h,h\rangle \geq \langle (I+R_2)g,g\rangle,
\end{equation*}
and using again \eqref{e:iau} we conclude that \eqref{e:pai} is equivalent to
\begin{equation}\label{e:xia}
\langle \xi,(I+R_1)^{-1}\xi\rangle \geq \langle \xi,(I+R_2)^{-1}\xi\rangle.
\end{equation}
Since, in \eqref{e:xia}, $\xi\in\cH$ is arbitrary, we have proven that
$(I+R_1)^{-1}\geq (I+R_2)^{-1}$, hence the direct implication in \eqref{e:aud} is proven. The
converse implication in \eqref{e:aud} follows as well, since all implications from above are
reversible.
\end{proof}

\begin{definition}[\cite{ReedSimon}]\label{d:order} 
Suppose $(R_n)_{n\in\NN}$ and $R$ are selfadjoint operators
on the same Hilbert space $\cH$.
The sequence $(R_n)_{n\in\NN}$ converges in the strong resolvent sense to $R$ if for every 
$\lambda\in\CC\setminus\RR$
\begin{equation*}
(\lambda I-R_n)^{-1}\xi\xrightarrow[n\ra\infty]{ } (I-R)^{-1}\xi,\quad \xi\in\cH.
\end{equation*}
\end{definition}

\begin{lemma}\label{l:tem} With notation as in Lemma~\ref{l:k2},
the mapping $\cT\ni\widetilde T\mapsto C^{-1}(\widetilde T)\in\cS$ defined at \eqref{e:ct} is
sequentially continuous when considering on $\cT$ the norm convergence and on $\cS$ the strong 
resolvent convergence.
\end{lemma}

\begin{proof} Suppose $(\widetilde T_n)_{n\in\NN}$ is a sequence of operators in $ \cT$ such that 
$\widetilde T_n\rightarrow \widetilde T\in\cT$ as $n\ra\infty$. Since
\begin{equation*}
(I+C^{-1}(\widetilde T))^{-1}=\frac{1}{2}(I+\widetilde T),
\end{equation*}
it follows that
\begin{equation}\label{e:iap}
(I+C^{-1}(\widetilde T_n))^{-1}\xi \rightarrow (I+C^{-1}(\widetilde T))^{-1}\xi,\quad \xi\in\cH.
\end{equation}
For every $h\in\dom(C^{-1}(\widetilde T))$ let the sequence $(h_n)_{n\in\NN}$, with elements in
$\dom(C^{-1}(\widetilde T))$, be defined by
\begin{equation*}
h_n=(I+C^{-1}(\widetilde T_n))^{-1}(I+C^{-1}(\widetilde T))h,\quad n\in\NN.
\end{equation*} 
By means of \eqref{e:iap} we get
\begin{equation}\label{e:hen}
h_n\xrightarrow[n\ra\infty]{ } (I+C^{-1}(\widetilde T))^{-1}(I+C^{-1}(\widetilde T))h=h,
\end{equation}
and, moreover,
\begin{align}\label{e:cel}
C^{-1}(\widetilde T_n)h_n & =C^{-1}(\widetilde T_n)(I+C^{-1}(\widetilde T_n))^{-1}
(I+C^{-1}(\widetilde T))h\\
& =h+C^{-1}(\widetilde T)h-h_n\xrightarrow[n\ra\infty]{ } C^{-1}(\widetilde T)h.\nonumber
\end{align}
From \eqref{e:hen}, \eqref{e:cel}, and \cite{ReedSimon}, VIII.~26, it follows that the sequence 
$(C^{-1}(\widetilde T_n))_{n\in\NN}$ converges in the strong resolvent sense to 
$C^{-1}(\widetilde T)$. 
\end{proof}

\section{The Main Theorem}

Let $T\colon\dom(T)\ra\cH$  be a symmetric contraction with $\dom(T)$ a closed subspace
of $\cH$ and consider the set, cf.\ \cite{Krein},
\begin{equation*}
\cB(T)=\{\widetilde T\in\cL(\cH)\mid \widetilde{T}\mbox{ is a selfadjoint contraction, } 
\widetilde T|
\dom(T)=T\},
\end{equation*}
where $\cL(\cH)$ denotes the algebra of all bounded linear operators $\cH\ra\cH$.
We present our argument for the fundamental result of M.G.~Kre\u\i n in \cite{Krein}.

\begin{theorem}\label{t:krein}
$\cB(T)\neq\emptyset$ and there exist $\widetilde T_{-1}$ and $\widetilde T_{1}$ in $\cB(T)$,
with $\widetilde T_{-1}\leq \widetilde T_1$, such that, if $B\in\cL(\cH)$ then
\begin{equation*} B\in\cB(T)\mbox{ if and only if }B=B^*\mbox{ and }\widetilde T_{-1}\leq B
\leq \widetilde T_1.
\end{equation*} 
\end{theorem}

\begin{proof}
The operator $T$ may be regarded as follows
\begin{equation*}
T\colon \dom(T)\ra \begin{matrix} \dom(T)\\ \oplus \\ \cH\ominus\dom(T)\end{matrix},
\end{equation*}
hence, by means of Lemma~\ref{l:fn}, $T=\left[\begin{matrix} A \\ \Gamma_2 D_A\end{matrix}\right]$
with $A\colon \dom(T)\ra\dom(T)$ a symmetric contraction and $\Gamma_2\colon\cD_A\ra\cH
\ominus \dom(T)$ a contraction, $A$ and $\Gamma_2$ being uniquely determined by $T$.

We search now for $\widetilde T\in\cB(T)$. Since $\widetilde T$ is selfadjoint and 
$\widetilde T|\dom(T)=T$, it must be of the following form
\begin{equation}\label{e:tt}
\widetilde T=\left[\begin{array}{ll} A & D_{A}\Gamma_2^* \\ \Gamma_2 D_A & X
\end{array}\right]
\end{equation}
with $X\colon\cH\ominus\dom(T)\ra\cH\ominus\dom(T)$ selfadjoint. Since $\widetilde T$ must
be a contraction as well, by Lemma~\ref{l:ag} and \eqref{e:tt}, we get
\begin{equation}\label{e:tte}
\widetilde T=\widetilde T(\Gamma)=\left[\begin{array}{ll} A & D_{A}\Gamma_2^* \\ \Gamma_2 D_A & 
-\Gamma_2 A^* \Gamma_2^*+D_{\Gamma_2^*}\Gamma D_{\Gamma_2^*}
\end{array}\right],
\end{equation}
with $\Gamma\colon \cD_{\Gamma_2^*}\ra\cD_{\Gamma_2^*}$ a selfadjoint contraction. 
The existence of at least one $\Gamma$ (for instance, $\Gamma=0$) proves that 
$\cB(T)\neq\emptyset$. If $\cD_{\Gamma_2^*}=\{0\}$ we take 
$\widetilde T_{-1}=\widetilde T_1=\widetilde T(0)$, that is, $\cB(T)$ has a single element.
If $\cD_{\Gamma_2^*}\neq \{0\}$ define
\begin{equation}\label{e:tt1}
\widetilde T_1=\widetilde T(I)=\left[\begin{array}{ll} A & D_{A}\Gamma_2^* \\ \Gamma_2 D_A & 
-\Gamma_2 A^* \Gamma_2^*+D_{\Gamma_2^*}^2
\end{array}\right]
=\left[\begin{array}{ll} A & D_{A}\Gamma_2^* \\ \Gamma_2 D_A & 
I-\Gamma_2 (I+A)^* \Gamma_2^*
\end{array}\right],
\end{equation}
\begin{equation}\label{e:ttm1}
\widetilde T_{-1}=\widetilde T(-I)=\left[\begin{array}{ll} A & D_{A}\Gamma_2^* \\ \Gamma_2 D_A & 
-\Gamma_2 A^* \Gamma_2^*-D_{\Gamma_2^*}^2
\end{array}\right]
=\left[\begin{array}{ll} A & D_{A}\Gamma_2^* \\ \Gamma_2 D_A & 
\Gamma_2(I- A)^* \Gamma_2^*-I
\end{array}\right].
\end{equation}
It is clear that $\widetilde T_{-1}\leq \widetilde T(\Gamma)\leq \widetilde T_1$ for any selfadjoint
contraction $\Gamma\colon\cD_{\Gamma_2^*}\ra\cD_{\Gamma_2^*}$.

Conversely, suppose $B\in\cL(\cH)$ is selfadjoint and
\begin{equation}\label{e:teb}
\widetilde T_{-1}\leq B\leq \widetilde T_1.
\end{equation}
Since $\widetilde T_{-1}$ and $\widetilde T_1$ are contractions it follows that $B$ 
itself is a contraction and, using again Lemma~\ref{l:ag}, it must have the form
\begin{equation*}
B=\left[\begin{array}{ll} C & D_{C}\Delta^* \\ \Delta D_C & 
-\Delta A^* \Delta^*+D_{\Delta^*}\Gamma^\prime D_{\Delta^*}
\end{array}\right],
\end{equation*}
with respect to the decomposition $\cH=\dom(T)\oplus(\cH\ominus\dom(T))$, where 
$C\colon\dom(T)\ra\dom(T)$ is a selfadjoint contraction, $\Delta\colon\cD(C)\ra
\cH\ominus\dom(T)$ is a contraction, and $\Gamma^\prime\colon\cD_{\Delta^*}\ra
\cD_{\Delta^*}$ is a selfadjoint contraction, all uniquely associated to $B$. If one makes use
of \eqref{e:teb} on $\dom(T)$ it follows that $C=A$. Now \eqref{e:teb} can be written as
\begin{equation}\label{e:bet1}
B-\widetilde T_{-1}=
\left[\begin{array}{ll} 0 & D_A(\Delta-\Gamma_2)^* \\
(\Delta-\Gamma_2)D_A & -\Delta A\Delta^* +D_{\Delta^*}\Gamma^\prime 
D_{\Delta^*} +\Gamma_2 A\Gamma_2^*+D_{\Gamma_2^*}^2\end{array}\right]
\geq 0,
\end{equation}
\begin{equation}\label{e:betm1}
\widetilde T_{1}-B=\left[\begin{array}{ll} 0 & D_A(\Gamma_2-\Delta)^* \\
(\Gamma_2-\Delta)D_A & -\Gamma_2 A\Gamma_2^*+D_{\Gamma_2^*}^2
+\Delta A\Delta^* -D_{\Delta^*}\Gamma^\prime 
D_{\Delta^*} \end{array}\right]\geq 0.
\end{equation}

At this point we have to recall that, for a given direct sum decomposition of $\cH$, if one considers
the operator $\left[\begin{array}{ll}  0 & M \\ M^* & N\end{array}\right]$ then
\begin{equation}\label{e:men}
\left[\begin{array}{ll}  0 & M \\ M^* & N\end{array}\right]\geq 0\mbox{ if and only if } M=0
\mbox{ and }N\geq 0.
\end{equation}
If one applies \eqref{e:men} to \eqref{e:bet1} and \eqref{e:betm1} it follows that 
$\Delta=\Gamma_2$ and, since \eqref{e:tte} is proven to be the general form of operators
from $\cB(T)$, we conclude that $B\in\cB(T)$.
\end{proof}

For a given symmetric contraction $T$, with notation as in \eqref{e:tte}, we consider the
set
\begin{equation}\label{e:ce}
\cC(T)=\{\Gamma\in\cL(\cD_{\Gamma_2^*})\mid \Gamma\mbox{ is a selfadjoint contraction}\}.
\end{equation}

\begin{proposition}\label{p:tem} Given $T$ a symmetric contraction, with notation as in \eqref{e:tte},
the mapping $\cC(T)\ni\Gamma\mapsto \widetilde T(\Gamma)\in\cB(T)$ as in \eqref{e:tte} is 
continuous, where $\cC(T)$ and $\cB(T)$ carry the operator
norm topologies, and such that, for any $\Gamma^\prime,\Gamma^{\prime\prime}\in\cC(T)$, 
\begin{equation*}
\Gamma^\prime\leq \Gamma^{\prime\prime}\mbox{ if and only if }\widetilde 
T(\Gamma^\prime)\leq\widetilde T(\Gamma^{\prime\prime}).
\end{equation*}
\end{proposition}

\begin{proof} If $\Gamma^\prime,\Gamma^{\prime\prime}\in\cC(T)$ then
\begin{equation}\label{e:tep}
\widetilde T(\Gamma^\prime)-\widetilde T(\Gamma^{\prime\prime})=
\left[\begin{array}{ll} 0 & 0 \\ 0 & D_{\Gamma_2^*}(\Gamma^\prime
-\Gamma^{\prime\prime})D_{\Gamma_2^*}\end{array}\right],
\end{equation}
and, considering the natural order relation for bounded selfadjoint operators, by \eqref{e:men}  it
now follows that  $\Gamma^\prime\leq \Gamma^{\prime\prime}$ if and only if $\widetilde 
T(\Gamma^\prime)\leq\widetilde T(\Gamma^{\prime\prime})$.

From \eqref{e:tep} we get
\begin{equation}\label{e:net}
\|\widetilde T(\Gamma^\prime)-\widetilde T(\Gamma^{\prime\prime})\|
=\|D_{\Gamma_2^*}(\Gamma^\prime-\Gamma^{\prime\prime})D_{\Gamma_2^*}\|,
\end{equation}
hence the continuity of the
mapping  $\cC(T)\ni\Gamma\mapsto \widetilde T(\Gamma)\in\cB(T)$ is clear.
\end{proof}

Consider now a positive, densely defined, and closed
operator $S$ in $\cH$ and let $T=C(S)$ defined
as in \eqref{e:cs}. We associate to $S$ the sets $\cS$ and $\cT$ as in Lemma~\ref{l:k2} and
clearly $\cT=\cB(T)$. By means of Lemma~\ref{l:k2} and the proof of Theorem~\ref{t:krein}, we
have obtained a bijective mapping $\cC(T)\ni\Gamma\mapsto C^{-1}(\widetilde T(\Gamma))\in \cS$.

On $\cC(T)$ we consider the natural order relation for bounded selfadjoint operators and the norm
topology and on $\cS$ we consider the order relation as in Definition~\ref{d:order} and the
strong resolvent convergence. From what we
have proven until now and the considerations from the preceding section we get

\begin{theorem}\label{t:bec} Given a positive, densely defined, and closed operator $S$ in $\cH$,
with notation as before, the bijective mapping 
$\cC(T)\ni\Gamma\mapsto C^{-1}(\widetilde T(\Gamma))=\widetilde S(\Gamma)\in \cS$ 
is sequentially continuous and
non-increasing, that is, for any $\Gamma^\prime,\Gamma^{\prime\prime}\in\cC(T)$ we
have
\begin{equation*}
\Gamma^\prime\leq\Gamma^{\prime\prime}\mbox{ if and only if }\widetilde S(\Gamma^\prime)\geq
\widetilde S(\Gamma^{\prime\prime}).
\end{equation*}
Moreover, there exist two positive selfadjoint extensions of $S$, $\widetilde S_\mathrm{K}=
\widetilde S(I)\leq \widetilde S(-I)=\widetilde S_\mathrm{F}$ such that, a positive selfadjoint operator
$R$ on $\cH$ belongs to $\cS$ if and only if 
$\widetilde S_\mathrm{K}\leq R\leq \widetilde S_\mathrm{F}$.
\end{theorem}

Consider $\Gamma\in\cC(T)$. Then $\dom(\widetilde S(\Gamma))=\ran(I+\widetilde T(\Gamma))$.
Since $\widetilde T(\Gamma)$ as in \eqref{e:tte} is defined as a $2\times 2$ block matrix 
corresponding to the decomposition $\cH=\ran(I+S)\oplus \ker(I+S^*)$ and $\dom(S)=\ran(I+T)$,
we get
\begin{equation}\label{e:domes}
\dom(\widetilde S(\Gamma))=\dom(S)+\left[\begin{array}{l} D_A\Gamma_2^* \\ I-\Gamma_2 A
\Gamma_2^*+D_{\Gamma_2^*}\Gamma D_{\Gamma_2}^*\end{array}\right]\ker(I+S^*),
\end{equation}
and for the moment this is all we can say about $\dom(\widetilde S(\Gamma))$. The next section
will improve the above formula, see Proposition~\ref{p:domg}.

\section{Special Semibounded Selfadjoint Extensions}

We have obtained in Theorem~\ref{t:krein} two remarkable positive selfadjoint extensions of $S$, 
$\widetilde S_{\mathrm{K}}=\widetilde S(I)$ and $\widetilde S_{\mathrm{F}}=\widetilde S(-I)$.
According to the general theory, \cite{AlonsoSimon}, \cite{Kato}, $\widetilde S_\mathrm{F}$ must
be the Friedrichs extension. $\widetilde S_\mathrm{K}$ was called in \cite{AlonsoSimon} the 
Kre\u\i n extension.

Let us denote by $F$ the Friedrichs extension of a positive, densely defined, closed operator $S$
in $\cH$. Then, \cite{Kato}, \cite{ReedSimon},
\begin{align}\label{e:domef}
\dom(F)=\{\xi\in\dom(S^*)\mid & \mbox{ there exists }(\xi_n)_{n\in\NN}\mbox{ in }\dom(S),\ 
\xi_n\xrightarrow[n\ra\infty]{\|\cdot\| }\xi\\
& \mbox{ and }\langle(\xi_n-\xi_m),S(\xi_n-\xi_m)\rangle\xrightarrow[m,n\ra\infty]{ } 0\},\nonumber
\end{align}
and
\begin{equation}\label{e:fe}
F\xi=S^*\xi,\quad \xi\in\dom(F).
\end{equation}
As before, we consider
\begin{equation}\label{e:cest}
C(S)=T=\left[\begin{matrix} A \\ \Gamma_2 D_A\end{matrix}\right]\colon\dom(T)\ra \begin{matrix}
\dom(T) \\ \oplus \\ \cH\ominus\dom(T),
\end{matrix}
\end{equation}
where $A\colon\dom(T)\ra\dom(T)$ is a symmetric contraction and 
$\Gamma_2\colon\cD_A\ra\cH\ominus\dom(T)$ is a contraction, uniquely determined by 
$T$ and hence by $S$.

\begin{lemma}\label{l:rana} 
$\ran(I+A)\subseteq \ran(I+T)^*\subseteq \ran((I+A)^{1/2})$, where the identity operator 
$I\colon\dom(T)\ra\dom(T)$ is identified with the embedding operator
$\colon\dom(T)\ra\dom(T)\subseteq \cH=\dom(T)\oplus(\cH\ominus\dom(T))$. 
\end{lemma}

\begin{proof} From \eqref{e:cest} it follows
\begin{equation}\label{e:iate}
I+T=\left[\begin{matrix} I+A \\ \Gamma_2 D_A \end{matrix}\right],\quad (I+T)^*=\left[I+A\ \ \ 
D_A\Gamma_2^*\right],
\end{equation}
therefore, $\ran(I+A)\subseteq \ran(I+T)^*$ and, one the other hand,
\begin{align*} (I+T)^*(I+T) & =(I+A)^2+D_A\Gamma_2^*\Gamma_2 D_A \\
& \leq (I+A)^2+D_A^2=(I+A)^2+(I-A^2)\\
& =2(I+A)=2(I+A)^{1/2}(I+A)^{1/2}.
\end{align*}
Making use of Theorem~1 in \cite{Douglas} we get $\ran(I+T)^*\subseteq \ran(I+A)^{1/2}$.
\end{proof}
\begin{lemma}\label{l:fex}
Suppose $\xi\in\dom(S^*)$. Then $\xi\in\dom(F)$ if and only if there exists a sequence 
$(\eta_n)_{n\in\NN}$ of vectors in $\dom(T)=\ran(I+S)$ such that
\begin{equation*}
(I+T)\eta_n\xrightarrow[n\ra\infty]{\|\cdot\|}\xi\mbox{ and }(I+A)^{1/2}(\eta_n-\eta_m)
\xrightarrow[m,n\ra\infty]{\|\cdot\|} 0.
\end{equation*}
\end{lemma}
\begin{proof} By means of \eqref{e:domef} and using
$\dom(S)=\ran(I+T)$ and $\dom(T)=\ran(I+S)$, an arbitrary 
vector $\xi\in\dom(S^*)$ belongs to $\dom(F)$ if and only if there
exists a sequence $(\eta_n)_{n\in\NN}$ of vectors in $\ran(I+S)$ such that
\begin{equation}\label{e:iat}
(I+T)\eta_n\xrightarrow[n\ra\infty]{\|\cdot\|}\xi,
\end{equation}
and
\begin{equation}\label{e:piat}
\langle(I+T)(\eta_n-\eta_m),(I-T)(\eta_n-\eta_m)\rangle\xrightarrow[m,n\ra\infty]{ } 0.
\end{equation}
Let us observe that \eqref{e:iat} yields
\begin{equation}\label{e:piata}
\langle(I+T)(\eta_n-\eta_m),(I+T)(\eta_n-\eta_m)\rangle\xrightarrow[m,n\ra\infty]{ } 0.
\end{equation}

Finally, letting $P$ denote the orthogonal projection of $\cH$ onto $\dom(T)$,
\begin{align*}
\|(I+A)^{1/2}(\eta_n-\eta_m)\|^2 & = \langle (I+A)(\eta_n-\eta_m),(\eta_n-\eta_m)\rangle\\
& = \langle P(I+T)(\eta_n-\eta_m),(\eta_n-\eta_m)\rangle\\
& = \langle (I+T)(\eta_n-\eta_m),
P(\eta_n-\eta_m)\rangle \\
& = \langle (I+T)(\eta_n-\eta_m),(\eta_n-\eta_m)\rangle \xrightarrow[m,n\ra\infty]{ }0,
\end{align*}
where, the convergence is obained 
by adding the quantities in \eqref{e:piat} and \eqref{e:piata}.
We have proven that $(I+A)^{1/2}(\eta_n-\eta_m)
\xrightarrow[m,n\ra\infty]{\|\cdot\|} 0$.
\end{proof}

\begin{proposition}\label{p:sef} $\widetilde S_\mathrm{F}=F$.
\end{proposition}

\begin{proof} Since both operators $\widetilde S_\mathrm{F}$ and $F$ are selfadjoint extensions 
of $S$, hence maximal symmetric, it is sufficient to prove 
$\dom(\widetilde S_\mathrm{F})\subseteq\dom(F)$.

To this end, let $\xi\in\dom(\widetilde S_\mathrm{F})=\ran(I+\widetilde T(-I))$. There exists 
$\eta\in\cH$ such that $\xi=(I+\widetilde T(-I))\eta$ and, since $I+T$ is one-to-one it follows that
$\ran(I+T)^*$ is dense in $\dom(T)$, hence, by Lemma~\ref{l:rana}, $\ran((I+A)^{1/2})$ is
dense in $\dom(T)$. Again by Lemma~\ref{l:rana}, there exists a sequence $(\eta_n)_{n\in\NN}$
of vectors in $\dom(T)$ such that
\begin{equation}\label{e:iaf}
(I+A)^{1/2}\eta_n\xrightarrow[n\ra\infty]{\|\cdot\|} (I+A)^{-1/2}(I+T)^*\eta,
\end{equation}
hence
\begin{equation}\label{e:iafa}
(I+A)\eta_n\xrightarrow[n\ra\infty]{\|\cdot\|} (I+T)^*\eta.
\end{equation}
Applying the bounded operator $(I-A)^{1/2}$ to \eqref{e:iaf} we get
\begin{equation*}
D_A\eta_n=(I-A)^{1/2}(I+A)^{-1/2}\eta_n\xrightarrow[n\ra\infty]{\|\cdot\|}
(I-A)^{1/2}(I+A)^{-1/2}(I+T)^*\eta,
\end{equation*}
and, since, considering $P$ the orthogonal
projection of $\cH$ onto ${\dom(T)}$ and using \eqref{e:iate}, we have
\begin{align*} 
(I-A)^{1/2}(I+A)^{-1/2}(I+T)^*\eta & = (I-A)^{1/2}(I+A)^{-1/2}
\bigl((I+A)P\eta+(I-A^2)^{1/2}\Gamma_2^*(I-P)\eta\bigr) \\
& = (I-A)^{1/2}(I+A)^{1/2}P\eta+(I-A)\Gamma_2^*(I-P)\eta \\
& = D_AP_\eta+(I-A)\Gamma_2^*(I-P)\eta,
\end{align*}
it follows that
\begin{equation}\label{e:gad}
\Gamma_2 D_A\eta_n\xrightarrow[n\ra\infty]{ } \Gamma_2 D_AP\eta+\Gamma_2(I-A)
\Gamma_2^*(I-P)\eta.
\end{equation}
From \eqref{e:iafa} and \eqref{e:gad} we have
\begin{align}\label{e:iaten}
(I+T)\eta_n=\left[\begin{matrix} (I+A)\eta_n \\ \Gamma_2 D_A\eta_n\end{matrix}\right]
\xrightarrow[n\ra\infty]{ } & \left[\begin{matrix} (I+A)P\eta+D_A\Gamma_2^* (I-P)\eta \\
\Gamma_2 D_A P\eta + \Gamma_2(I-A)\Gamma_2^* (I-P)\eta\end{matrix}\right] \\
& = (I+\widetilde T(-I))\eta=\xi. \nonumber
\end{align}
Finally, from \eqref{e:iaf}, \eqref{e:iaten}, and Lemma~\ref{l:fex} we obtain $\xi\in\dom(F)$.
\end{proof}

As a consequence we can determine the domain of an arbitrary positive selfadjoint extension
$\widetilde S(\Gamma)$ in terms of the domain of the Friedrichs extension and the parameter 
$\Gamma$.

\begin{proposition}\label{p:domg}
For every $\Gamma\in\cC(T)$ we have 
\begin{equation*}
\dom(\widetilde S(\Gamma))=\dom(F)+D_{\Gamma_2^*}(I+\Gamma)D_{\Gamma_2^*}\ker(I+S^*).
\end{equation*}
\end{proposition}

\begin{proof} For arbitrary $\Gamma\in\cC(T)$, by \eqref{e:tte} and \eqref{e:ttm1}, we have
\begin{align*}
I+\widetilde T(\Gamma) & = \left[\begin{array}{ll} I+ A & D_{A}\Gamma_2^* \\ \Gamma_2 D_A &
\Gamma_2 (I-A)^* \Gamma_2^*+D_{\Gamma_2^*}(I+\Gamma) D_{\Gamma_2^*}
\end{array}\right]\\
& = I+\widetilde T(-I)+ \left[\begin{array}{ll} 0 & 0 \\ 0 &
D_{\Gamma_2^*}(I+\Gamma) D_{\Gamma_2^*}
\end{array}\right],
\end{align*}
hence, from Proposition~\ref{p:sef}, we get
\begin{align*}
\dom(\widetilde S(\Gamma)) & = \ran(I+\widetilde T(\Gamma)) \\
& = \ran(I+\widetilde T(-I))+D_{\Gamma_2^*}(I+\Gamma)D_{\Gamma_2^*} (\cH\ominus\dom(T))\\
& = \dom(F)+D_{\Gamma_2^*}(I+\Gamma)D_{\Gamma_2^*}\ker(I+S^*).\qedhere
\end{align*}
\end{proof}

\end{document}